\documentclass[11pt]{article}

\usepackage{soul, color}
\definecolor{linkcolor}{rgb}{0.25,0.25,1}
\usepackage[pagebackref,colorlinks=true,pdfpagemode=none,urlcolor=linkcolor,linkcolor=linkcolor,citecolor=linkcolor]{hyperref}
\usepackage{empheq}
\usepackage{makeidx}
\usepackage{amsthm, amssymb, bm}
\usepackage[]{amsmath}
\usepackage[]{amsfonts}
\usepackage[]{fancyhdr}
\usepackage[]{graphicx}
\usepackage{algorithm}
\usepackage{algorithmic}
\usepackage{parsl}

\graphicspath{{/EPSF/}{../figures/}{figures/}}

\newtheorem{theorem}{Theorem}[section]
\newtheorem{lemma}[theorem]{Lemma}
\newtheorem{proposition}[theorem]{Proposition}
\newtheorem{corollary}[theorem]{Corollary}
\newtheorem{definition}[theorem]{Definition}
\newtheorem{remark}[theorem]{Remark}

\newtheorem{problem}{Problem}[section]

\def \Hm {\mathbb{H}}
\def \Imm {\mathbb{I}}

\def \Rm {\mathbb{R}}

\def\A{\mathcal{A}}

\def\C{\mathcal{C}}
\def\D{\mathcal{D}}
\def\F{\mathcal{F}}
\def\G{\mathcal{G}}
\def\H{\mathcal{H}}
\def\I{\mathcal{I}}

\def\M{\mathcal{M}}
\def\N{\mathcal{N}}
\def\O{\mathcal{O}}

\def\S{\mathcal{S}}
\def\V{\mathcal{V}}

\renewcommand{\div}{\nabla\cdot} 
 
\newcommand{\del}{\overline{\nabla}}

\newcommand{\wtA}{ {\widetilde A} }

\newcommand{\where}{\quad\text{ where }}
\newcommand{\qandq}{\quad\text{ and }\quad}

\newcommand{\rank}{\text{rank}}

\newcommand{\bfe}{ {\bf e}}

\newcommand{\bfg}{ {\bf g}}
\newcommand{\bfh}{ {\bf h}}

\newcommand{\tr}{ {\text{tr }}}

\newcommand{\cout}[1]{}

\newcommand{\dprod}[2]{\langle #1, #2 \rangle }

\def \dist { {\mbox{dist }}}

\title{Inverse anisotropic conductivity from power densities in dimension $n\ge 3$}
\author{Fran\c cois Monard\thanks{Department of Applied Physics and Applied Mathematics, Columbia University,  New York NY, 10027; fm2234@columbia.edu} \and Guillaume Bal\thanks{Department of Applied Physics and Applied Mathematics, Columbia University,  New York NY, 10027; gb2030@columbia.edu}}

\begin{document}

\maketitle

\begin{abstract} 
    We investigate the problem of reconstructing a fully anisotropic conductivity tensor $\gamma$ from internal functionals of the form $\nabla u\cdot\gamma\nabla u$ where $u$ solves $\nabla\cdot(\gamma\nabla u) = 0$ over a given bounded domain $X$ with prescribed Dirichlet boundary condition. This work motivated by hybrid medical imaging methods covers the case $n\ge 3$, following the previously published case $n=2$ \cite{Monard2011}. Under knowledge of enough such functionals, and writing $\gamma = \tau \tilde \gamma$ ($\det \tilde\gamma = 1$) with $\tau$ a positive scalar function, we show that all of $\gamma$ can be explicitely and locally reconstructed, with no loss of scales for $\tau$ and loss of one derivative for the anisotropic structure $\tilde\gamma$. The reconstruction algorithms presented require rank maximality conditions that must be satisfied by the functionals or their corresponding solutions, and we discuss different possible ways of ensuring these conditions for $\C^{1,\alpha}$-smooth tensors ($0<\alpha<1$).    
\end{abstract}

\section{Introduction}
Hybrid medical imaging methods aim to combine a high-resolution modality (such as acoustic waves or Magnetic Resonance Imaging) with a high-constrast one (e.g. Electrical Impedance Tomography, Optical Tomography, \dots) in order to improve the result of the latter thanks to a physical coupling. In this context, the problem we consider is motivated by a coupling between an elliptic equation (modelling conductivity or stationary diffusion) and acoustic waves. Namely we consider the problem of reconstructing a fully anisotropic conductivity (or diffusion) tensor $\gamma$ over a domain of interest $X\subset\Rm^n$ from knowledge of a certain number of {\em power density functionals} of the form $\H_\gamma [u](x) = \nabla u\cdot \gamma\nabla u (x)$, where $u$ solves the following partial differential equation
\begin{align}
    -\nabla\cdot(\gamma\nabla u) = -\sum_{i,j=1}^n \partial_i (\gamma^{ij} \partial_j u) = 0\quad (X), \quad u|_{\partial X} = g,
    \label{eq:conductivity}
\end{align}
where the boundary condition $g$ is prescribed. By polarization, we will see that {\em mutual power densities} of the form $\nabla u\cdot\gamma\nabla v$ will also be considered, where both $u$ and $v$ solve \eqref{eq:conductivity} with different boundary conditions. The model above, when considered as a diffusion model for photons in tissues, should be augmented with a term $\sigma_a u$ accounting for absorption and will be addressed in future work. The availability of such functionals is justified by a coupling with acoustic waves, as it is described in the context of Ultrasoud Modulated- Electrical Impedance Tomography (UMEIT) or Optical Tomography (UMOT) in \cite{Ammari2008,Bal2010d,Kuchment2010,Bal2011a} by considering acoustic deformations, or in the context of Impedance-Acoustic Computerized Tomography in \cite{Gebauer2009} by considering thermoelastic effects. In both cases, the acoustic waves come to the rescue of an otherwise very ill-posed problem (the classical Calder\'on's problem of recovering $\gamma$ from its Dirichlet-to-Neuman operator, see \cite{calderon}), by providing internal functionals instead of boundary ones. In \eqref{eq:conductivity}, we require $\gamma$ to have bounded components and to be {\em uniformly elliptic} as defined by the following condition
\begin{align}
    |\xi|^2 \kappa^{-1} \le \gamma(x)\xi\cdot\xi \le \kappa |\xi|^2, \quad x\in X, \xi\in\Rm^n,
    \label{eq:ellipticity}
\end{align}
from some $\kappa\ge 1$. Borrowing notation from \cite{Astala2005}, we denote $C(\gamma)$ the smallest such constant $\kappa$ and define the set
\begin{align}
    \Sigma(X) := \{ \gamma \in L^\infty(X), \quad C(\gamma)<\infty  \}.
    \label{eq:SigmaX}
\end{align}
With these definitions, our problem may be formulated as follows
\begin{problem}[Inverse conductivity from power density functionals]\label{prob:UMEIT}
    For $\gamma$ in $\Sigma(X)$ or any subset of it, does the power density measurement operator $\H_\gamma$ uniquely characterize $\gamma$ ? If yes, how stably ?
\end{problem}


The problem just described has received fair attention in the past few years. The first inversion formula for Problem \ref{prob:UMEIT} was given in \cite{Capdeboscq2009} in the isotropic, two-dimensional setting. There, a constructive algorithm as well as an optimal control approach for numerical reconstruction were presented. \cite{Kuchment2011a} then studied a linearized, isotropic version of Problem \ref{prob:UMEIT} in dimensions two and three with numerical implementation. 


Problem \ref{prob:UMEIT} has also been studied under constraints of limitations on the number of power densities available, the most restrictive case being the reconstruction of an isotropic tensor $\gamma = \sigma\Imm_n$ in \eqref{eq:conductivity} from only one measurement $H = \sigma|\nabla u|^2$. In this case, $\sigma$ may be replaced in \eqref{eq:conductivity} by $H/|\nabla u|^2$, and this yields the following non-linear partial differential equation 
\begin{align*}
    \nabla\cdot \left( \frac{H}{ |\nabla u|^2} \nabla u \right) = 0 \quad (X), \quad u|_{\partial X} = g.
\end{align*}
Newton-based methods were proposed in \cite{Gebauer2009} in order to successively reconstruct $u$ and $\sigma$, and the corresponding Cauchy problem was studied theoretically in \cite{Bal2012d}.


In search for explicit reconstruction formulas using larger numbers of functionals, the authors first extended the reconstruction result from \cite{Capdeboscq2009} to the three-dimensional, isotropic case in \cite{Bal2011a} with Bonnetier and Triki. This result was then generalized in \cite{Monard2011a} to $n$-dimensional, isotropic tensors with more general types of measurements of the form $\sigma^{2\alpha} |\nabla u|^2$ with $\alpha$ not necessarily $\frac{1}{2}$. This covers the case $\alpha=1$ of Current Density Impedance Imaging \cite{Nachman2011,Seo2011}. Finally, the same authors derived reconstruction formulas for the fully anisotropic two-dimensional problem and validated them numerically in \cite{Monard2011}.

In the last three papers presented, the explicit reconstruction algorithms were derived in the case where the power densities belong to $W^{1,\infty}(X)$, and assuming some qualitative properties satisfied by the solutions. In particular, the reconstruction algorithm for the isotropic case (or, equivalently, of a scalar factor multiplied by a known anisotropic tensor) strongly relies on the existence of $n$ solutions of \eqref{eq:conductivity} whose gradients form a basis of $\Rm^n$ at every point of the domain. Under such assumptions, stability estimates were derived for the reconstruction schemes proposed, of Lipschitz type for the determinant of the conductivity tensor under knowledge of the anisotropic structure $\tilde \gamma := (\det\gamma)^{-\frac{1}{n}} \gamma$, and of (less stable) H\"older type for the anisotropic structure $\tilde \gamma$. Finally, it was shown for certain types of tensors $\gamma$ that the assumption of linear independence made on the solutions could be guaranteed {\it a priori} by choosing appropriate boundary conditions, so that all the reconstruction procedures previously established could be properly implemented.


Studying a linearized version of Problem \ref{prob:UMEIT} from the pseudo-differential calculus standpoint, the Lipschitz stability mentioned above was also pointed out in \cite{Kuchment2011} in the isotropic case. There, the authors showed that from three power densities functionals, the linearized power density operator is an elliptic functional of an isotropic tensor $\sigma$. They also studied in more detail the ``stabilizing'' nature of internal functionals of certain kinds that have arisen in all the hybrid medical imaging methods mentioned above. An extension of this result to the anisotropic case is presently investigated by the authors with Guo in \cite{Bal2012e}.

The present work aims at unifying and extending the work done in \cite{Capdeboscq2009,Bal2011a,Monard2011a,Monard2011} by treating in full extent the anisotropic, $n$-dimensional case of Problem \ref{prob:UMEIT} for $\C^{1,\alpha}$-smooth conductivities with $0<\alpha<1$ (the H\"older exponent is required by forward elliptic theory). The basis of this work also appears and will strongly rely on the first author's recent thesis \cite{Monard2012b}.

\section{Statement of the main results} \label{sec:statement}

We decompose the conductivity tensor $\gamma$ into the product of a scalar factor $\tau := (\det \gamma)^{\frac{1}{n}}$ and a scaled anisotropic structure $\tilde\gamma$:
\begin{align}
    \gamma := \tau \tilde\gamma, \qquad \tau:= (\det\gamma)^{\frac{1}{n}}, \quad \det\tilde\gamma = 1.
    \label{eq:gammadecomp}
\end{align}
Note that when $\gamma\in\Sigma(X)$, $\tau$ is uniformly bounded above and below by $C(\gamma)$ and $C(\gamma)^{-1}$, respectively.

Under knowledge of enough power densities inside the domain, the reconstructibility of $\tau$ and/or $\tilde\gamma$ are local questions, since under certain conditions described below, both quantities $\tau$ and $\tilde\gamma$ can be explicitely and locally recovered in terms of power densities and their derivatives. We first describe these conditions and the corresponding recosntruction formulas in the next paragraph. 

Second, as the reconstruction algorithms presented above require local conditions, we will describe how to control these conditions from the domain's boundary, also tackling the question of global reconstruction. 

\subsection{Local reconstruction algorithms}

\paragraph{Reconstruction of the scalar factor $\tau$ knowing $\tilde\gamma$: }

We first consider the question of local reconstructibility of the scalar factor $\tau$ under knowledge of the anisotropic structure $\tilde\gamma$. The main hypothesis here is that we use the mutual power densities $H_{ij} := \nabla u_i\cdot\gamma\nabla u_j$ (for $1\le i,j\le n$) of $n$ solutions $(u_1,\dots,u_n)$ of \eqref{eq:conductivity} whose gradients are {\em linearly independent} over a subdomain $\Omega\subset X$, a condition which we formulate as 
\begin{align}
    \inf_{x\in\Omega} |\det (\nabla u_1,\dots,\nabla u_n)| \ge c_0 >0.
    \label{eq:positivityn}
\end{align}
Under this assumption we are able to derive the following reconstruction formula: defining $A = \gamma^{\frac{1}{2}}$ to be the positive matrix squareroot of $\gamma$, and decomposing $A$ into $A = \sqrt{\tau} \wtA$ with $\det\wtA = 1$, knowledge of $\tilde\gamma$ implies knowledge of $\wtA$. Further defining $S_i:= A\nabla u_i$ for $1\le i\le n$, the data becomes $H_{ij} := S_i\cdot S_j$. Such vector fields satisfy the following PDE's
\begin{align}
    d (\wtA^{-1}S_i)^\flat = d\log\tau \wedge (\wtA^{-1}S_i)^\flat \qandq \nabla\cdot (\wtA S_i) = - \nabla\log\tau\cdot\wtA S_i, \quad 1\le i\le n,
    \label{eq:PDESi}
\end{align}
where the equality of two-forms expresses the fact that $d (A^{-1} S_i)^\flat = d^2 u_i = 0$ (exact forms are closed), and the scalar equality is deduced from the conductivity equation. Here the $^\flat$ exponent denotes the {\em flat} (or index-lowering) operator for the Euclidean metric. From these PDE's, one can derive the following formula, first established in \cite[Lemma 4.3.1]{Monard2012b} as a generalization of earlier results in \cite{Bal2011a,Capdeboscq2009,Monard2011,Monard2011a}:
\begin{align}
    \nabla\log\tau = \frac{2}{n} |H|^{-\frac{1}{2}} \left( \nabla (|H|^{\frac{1}{2}}H^{jl})\cdot\wtA S_l \right)\wtA^{-1} S_j = \frac{1}{n} \nabla\log |H| + \frac{2}{n} (\nabla H^{jl}\cdot\wtA S_l)\wtA^{-1}S_j, \quad x\in\Omega. 
    \label{eq:gradlogtau}
\end{align}

Equation \eqref{eq:gradlogtau} may thus be used to substitute $\nabla\log\tau$ into the PDE's \eqref{eq:PDESi}, and the resulting system becomes closed for the frame $S \equiv (S_1,\dots,S_n)$. We then show that such a system may be rewritten as a first-order quasilinear system of the form 
\begin{align}
    \del S_i = \S_i (S, H, dH, \wtA, d\wtA), \quad 1\le i\le n, \quad x\in \Omega,
    \label{eq:gradSi}
\end{align}
where $\S_i$ is a Lipschitz functional of the components of the frame $S$. Here $\del S_i$ denotes the total covariant derivative of the vector field $S_i$, a tensor field of type $(1,1)$ that encodes all partial derivatives $\partial_p S_i^q$. System \eqref{eq:gradSi} can thus be integrated over any curve to reconstruct the value of $S$ from knowledge of $S(x_0)$ for fixed $x_0\in \Omega$. Once $S$ is known throughout $\Omega$, $\tau$ can be reconstructed throughout $\Omega$ by integrating \eqref{eq:gradlogtau} in a similar fashion. The PDE's \eqref{eq:gradlogtau} and \eqref{eq:gradSi} are overdetermined and come with compatibility conditions which should hold as long as our measurements are in the range of the measurement operator. In such a case, this leads to a unique and stable reconstruction in the sense of the following proposition, first stated in \cite[Prop. 4.3.6-4.3.7]{Monard2012b}:

\begin{proposition}[Local stability for $\log\tau$]\label{prop:stabtau}
    Consider two tensors $\gamma = \tau \wtA^2$ and $\gamma' = \tau' \wtA^{'2}$ in $\Sigma(X)$, where $\wtA$ and $\wtA'$ are known and with components in $W^{1,\infty}(X)$. Let $\Omega\subset X$ such that the positivity \eqref{eq:positivityn} holds for two sets of conductivity solutions $(u_1,\cdots,u_n)$ and $(u'_1,\cdots,u'_n)$ with respective conductivities $\gamma$ and $\gamma'$, call their corresponding data sets $\{H_{ij}, H'_{ij}\}$ with components in $W^{1,\infty}(X)$. Then the functions $\log\tau$ and $\log\tau'$ can be uniquely reconstructed with the following stability estimate
    \begin{align*}
	\|\log \tau - \log \tau'\|_{W^{1,\infty}(\Omega)} \le \varepsilon_0 + C\left(\|H-H'\|_{W^{1,\infty}(X)} + \|\wtA-\wtA'\|_{W^{1,\infty}(X)}\right),
    \end{align*}
    where the constant $C$ does not depend on $\Omega$ and $\varepsilon_0$ is the error committed at some $x_0\in\Omega$.    
\end{proposition}
Such a stability statement shows that under the condition \eqref{eq:positivityn}, the reconstruction of $\tau|_\Omega$ is a {\em well-posed} problem in $W^{1,\infty}(\Omega)$. Section \ref{sec:tau} contains the proofs of equations \eqref{eq:gradlogtau} and \eqref{eq:gradSi}. 

\paragraph{Reconstruction of the anisotropic structure $\tilde\gamma$, then of $\tau$: }

Here and below, we denote by $M_n(\Rm)$ the space of $n\times n$ matrices with its inner product structure $\dprod{A}{B} := A_{ij}B_{ij} = \tr (AB^T)$. We now derive an approach to reconstruct the anisotropic structure $\tilde\gamma$ from additional measurements. We start from a basis of solutions $(u_1,\dots,u_n)$ satisfying \eqref{eq:positivityn} over $\Omega\subset X$. Considering an additional conductivity solution $v$, we show that, although the solutions $(u_1,\dots,u_n,v)$ are themselves unknown, the decomposition of $\nabla v$ in the basis $(\nabla u_1,\dots,\nabla u_n)$ is {\em known from the power densities}. Combining these equations with the PDE's satisfied by the solutions allows to derive linear orthogonality constraints on the product matrix $\wtA S$ ($S$ here denotes the matrix with columns $S_1,\dots,S_n$). 

Thus, any additional solution $v$, by means of its power densities with the support basis, gives rise to a subspace $\V \subset M_n(\Rm)$ orthogonal to $\wtA S$, moreover a basis of $\V$ is known from the data. The dimension of $\V$ is accurately controlled in \cite[Prop. 4.3.8]{Monard2012b} and its maximal value is 
\begin{align}
    \dim \V\le d_M := 1 + n(n+1)/2. 
    \label{eq:dm}
\end{align}
The matrix $\wtA S$ is arbitrary in $M_n(\Rm)$ except for its determinant, known up to sign, thus $\wtA S$ requires $n^2-1$ independent constraints to be determined up to sign. This requires that we consider enough additional solutions $v_1,\dots,v_l$ such that their corresponding spaces $\V_1,\dots,\V_l$ satisfy (i) $\dim \sum_{i=1}^l \V_i = n^2-1$, and (ii) $\left(\sum_{i=1}^l \V_i\right)^\perp$ is spanned by a non-singular matrix (this condition should always hold true when measurements aren't noisy, as this orthogonal space is nothing but $\Rm \wtA S$). In mathematical terms, the proper condition to satisfy is as follows: for $1\le i\le l$, let $M_{(i)1},\dots,M_{(i)d_M}$ span $\V_i$ (they can be constructed from the data), and denote
\begin{align}
    \M := \{ M_{(i)j}\ |\ 1\le i\le l, \quad 1\le j\le d_M  \}, \quad \#\M = d_M l,
    \label{eq:Mfam}
\end{align}
rewritten more simply as $\M = \{M_i\ |\ 1\le i\le d_M l \}$ below. Conditions (i) and (ii) mentioned above will hold if for $x\in\Omega$, there exists an $n^2-1$-subfamily of $\M$ with nonzero hypervolume. With the notion of cross-product in Appendix \ref{app:crossprod}, this condition may be written under the form
\begin{align}
    \inf_{x\in \Omega} \sum_{I\in\I(n^2-1,\#\M)} (\det (\N(I)H^{-1}\N(I)))^\frac{1}{n} \ge c_1 >0,
    \label{eq:hyperplane}
\end{align}
for some constant $c_1$, where $\I(n^2-1,d_M l)$ denotes the set of increasing injections from $[1,n^2-1]$ to $[1,d_M l]$ (i.e. $I\in \I(n^2-1,d_M l)$ is of the form $I = (i_1,\dots,i_{n^2-1})$ with $1\le i_1<\dots<i_{n^2-1}\le d_M l$), and where $\N(I) = \N(M_{i_1},\dots,M_{i_{n^2-1}})$ is the cross-product defined in Appendix \ref{app:crossprod}. Under condition \eqref{eq:hyperplane}, we are able to reconstruct $\tilde\gamma$ and $\nabla \log\tau$ via formulas \eqref{eq:gammarecons} and \eqref{eq:gradlogtaurecons}. This reconstruction is unique and stable in the sense of the proposition below. 

\begin{proposition}[Local stability for $\tilde\gamma$ and $\log\tau$]\label{prop:stabgamma}
    Consider two tensors $\gamma = \tau \tilde\gamma$ and $\gamma' = \tau' \tilde\gamma'$ in $\Sigma(X)$. Let $\Omega\subset X$ where $u_1,\dots,u_n, v_1,\dots,v_l$ and $u_1',\dots,u_n', v_1',\dots,v_l'$ satisfy conditions \eqref{eq:positivityn} and \eqref{eq:hyperplane}. Then $\gamma$ and $\gamma'$ are uniquely reconstructed from knowledge of the power densities of the above sets of solutions, and we have the following stability estimate
    \begin{align}
	\|\nabla(\log\tau-\log\tau')\|_{L^{\infty}(\Omega)} + \|\tilde\gamma-\tilde\gamma'\|_{L^\infty(\Omega)} \le C \|H-H'\|_{W^{1,\infty}(\Omega)}.
	\label{eq:stabgamma}
    \end{align}
\end{proposition}

\begin{remark}
    Although Proposition \ref{prop:stabtau} required bounded derivatives on the anisotropic structures $\tilde\gamma$, this is no longer the case here as the frame $S$ is reconstructed algebraically instead of solving a differential system that involves derivatives of the anisotropic structure. This is in good agreement with the fact that the stability statement \eqref{eq:stabgamma} is only stated in $L^\infty$-norm for $\tilde\gamma$. 
\end{remark}

\begin{remark}
    The scalar factor $\tau$ is reconstructed with better stability than the anisotropic structure $\tilde\gamma$, for which there is locally a loss of one derivative. Although the reconstruction procedure presented was not yet proven optimal in terms of number of power densities involved, this loss of one derivative cannot be avoided and finds justification in the microlocal analysis of the linearized problem that will appear on future work \cite{Bal2012e}.
\end{remark}

\paragraph{Local reconstructibility:} In the light of the local reconstruction algorithms previously derived, a tensor $\gamma\in\Sigma(X)$ is {\em locally reconstructible from power densities} if for every $x\in X$, there exists a neighborhood $\Omega_x\ni x$ and $n+l$ boundary conditions $(g_1,\dots,g_n,h_1,\dots,h_l)$ such that the corresponding $n$ first conductivity solutions satisfy condition \eqref{eq:positivityn} and the $l$ remaining ones satisfy condition \eqref{eq:hyperplane} (which then ensures via Proposition \ref{prop:stabgamma} that $\gamma$ is uniquely and stably reconstructible over $\Omega_x$). Based on the Runge approximation for elliptic equations \cite{Lax1956}, we then have the following generic result:

\begin{theorem}[Local reconstructibility of $\C^{1,\alpha}$ tensors, $\alpha>0$] \label{thm:local}
    If $\gamma\in \C^{1,\alpha}(\overline{X})$, then $\gamma$ is locally reconstructible from power densities.     
\end{theorem}

\begin{remark}
    In a similar manner to \cite{Bal2012}, the proof of Theorem \ref{thm:local} is based on constructing solutions locally, that will fulfill conditions \eqref{eq:positivityn}-\eqref{eq:hyperplane}, after which such solutions will be approximated using the Runge approximation by solutions of \eqref{eq:conductivity} globally defined over $X$ and controlled from the boundary. Although this result establishes local reconstructibility for a large class of tensors, the applicability remains limited insofar as the boundary conditions are not explicitely constructed. 
\end{remark}

\subsection{Global reconstructions}

The previous approach consisted in deriving explicit reconstruction algorithms under certain {\it a priori} conditions (linear independence or rank maximality) satisfied by a certain number of solutions of \eqref{eq:conductivity}. These conditions may be checked directly on the power densities at hand. As the user only has control over boundary conditions in this problem, it is thus appropriate to define sets of {\em admissible boundary conditions}, for which the conditions mentioned above are satisfied globally. 

The first admissibility set is that of $m$-tuples of boundary conditions ($m\ge n$) such that, locally, $n$ of the $m$ solutions of \eqref{eq:conductivity} have linearly independent gradients. This ensures that the scalar factor $\tau$ is uniquely and stably reconstructible throughout the domain. For $\gamma\in\Sigma(X)$, we call such an admissibility set $\G_\gamma^m$ ($m\ge n$), subset of $(H^{\frac{1}{2}}(\partial X))^m$, whose full definition is given in Def. \ref{def:Ggammam}.

On to the global reconstruction of $(\tilde\gamma, \tau)$, Definition \ref{def:Agammamk} constructs a second set of admissible boundary conditions. We first pick $\bfg\in\G_\gamma^m$ for some $m\ge n$ so that a basis of gradients of solutions may be available everywhere. Considering $l\ge 1$ additional solutions generated by boundary conditions $\bfh = (h_1,\dots,h_l)$, we say that $\bfh$ belongs to $\A_\gamma^{m,l}(\bfg)$ if the spaces $(\V_1,\dots,\V_l)$, generated by $(v_1,\dots,v_l)$ as in the previous section, form everywhere a hyperplane of $M_n(\Rm)$, so that $\tilde\gamma$ and $\tau$ may be reconstructed with the stability of Proposition \ref{prop:stabgamma}.

While the construction of these sets is somewhat tedious, they allow us to define sufficient conditions for global reconstructibility of all or part of $\gamma$. In particular, they allow us to reformulate a reconstructibility statement into a non-emptiness statement on sets of admissible boundary conditions $\G_\gamma$ or $\A_\gamma$, which are characterized by continuous functionals of power densities. Namely, for a tensor $\gamma = \tau\tilde\gamma$, we have 

\begin{theorem}[Global reconstructibility] \label{thm:reconstructibility}
    \begin{enumerate}
	\item Under knowledge of a $\C^1$-smooth $\tilde\gamma$, the function $\tau\in W^{1,\infty}(X)$ is uniquely reconstructible if $\G_{\gamma}^m \ne\emptyset$ for some $m\ge n$, with a stability estimate of the form
	    \begin{align}
		\|\log\tau-\log\tau'\|_{W^{1,\infty}(X)} \le C \left( \|H-H'\|_{W^{1,\infty}(X)} + \|\tilde\gamma-\tilde\gamma'\|_{W^{1,\infty}(X)} \right).
		\label{eq:stabtauglobal}
	    \end{align}
	\item $\gamma$ is uniquely reconstructible if there exists $m\ge n$ and $l\ge 1$ such that $\G_{\gamma}^m\ne \emptyset$ and $\A_{\gamma}^{m,l}(\bfg)\ne\emptyset$ where $\bfg\in\G_\gamma^m$, with the stability estimate
	    \begin{align}
		\|\tilde\gamma-\tilde\gamma'\|_{L^\infty(X)} + \|\log\tau - \log\tau'\|_{W^{1,\infty}(X)} \le C \|H-H'\|_{W^{1,\infty}(X)}. 
		\label{eq:stabgammaglobal}
	    \end{align}
    \end{enumerate}
\end{theorem}

Combining this with the fact that the conditions of linear independence stated above can be formulated in terms of continuous functionals of power densities and their derivatives, we can deduce further useful facts about the sets $\G_\gamma$ and $\A_\gamma$, all of which are established in \cite[Sec. 5.2.1]{Monard2012b}, allowing us to draw the following conclusions (see Sec. \ref{sec:admsetprop}):
\begin{itemize}
    \item The reconstruction algorithms presented above remain stable under $\C^2$-smooth perturbations of the boundary conditions. 
    \item Conductivity tensors that are close enough in $\C^{1,\alpha}$ norm can be reconstructed from power densities emanating from the same boundary conditions. 
    \item The property of being reconstructible carries through push-forwards of conductivity tensors by diffeomorphisms, see in particular Proposition \ref{prop:pushfwd} below.
\end{itemize}

With these properties in mind, global reconstructibility is thus established for conductivity tensors that are $\C^{1,\alpha}$-close to or push-fowarded from the cases below:
\begin{enumerate}
    \item If $\gamma = \tau \Imm_n$ for some scalar function $\tau\in H^{\frac{n}{2}+3+\varepsilon}$ with $\varepsilon>0$, then $\G_\gamma^n\ne\emptyset$ for $n$ even and $\G_\gamma^{n+1}\ne\emptyset$ for $n$ odd. The proof can be found in \cite{Monard2011a} and relies on the construction of complex geometrical optics solutions.  
    \item For $\gamma = \Imm_n$, straighforward computations (see Sec. \ref{sec:prooflocal} below) show that $Id|_{\partial X} \in \G_\gamma^n$ and $\{x_i^2-x_{i+1}^2\}_{i=1}^{n-1}|_{\partial X} \in \A_{\gamma}^{n,n-1} (Id|_{\partial X})$. From this observation, one can cover the case of a constant tensor $\gamma_0$ by pushing forward the above solutions with the diffeomorphism $\Psi(x) = \gamma_0^{-\frac{1}{2}} x$.
\end{enumerate}

\paragraph{Outline:} The rest of the paper is organized as follows. Section \ref{sec:local} justifies the local reconstruction algorithms. For the reconstruction of $\tau$, Section \ref{sec:tau} provides proof of equations \eqref{eq:gradlogtau} and \eqref{eq:gradSi}, Section \ref{sec:tildegamma} covers the reconstruction of all of $\gamma$, while Section \ref{sec:prooflocal} concentrates on proving Theorem \ref{thm:local}. On to the question of global reconstructibility, Section \ref{sec:global} first studies the properties of the admissibility sets $\G_\gamma$ and $\A_\gamma^{m,l}$ before discussing what tensors are globally reconstructible.


\section{Local reconstruction formulas} \label{sec:local}

\subsection{Reconstruction of the scalar factor $\tau$} \label{sec:tau}

\paragraph{Geometric setting and preliminaries:}
We equip $X$ with the Euclidean metric $g(U,V) \equiv U\cdot V = \delta_{ij} U^i V^j$, where the Einstein summation convention is adopted. For $x\in X$, $(\bfe_1,\dots,\bfe_n)$ and $(\bfe^1,\dots,\bfe^n)$ denote the canonical bases of $T_x X$ and $T_x^\star X$, respectively. The flat operator coming from the Euclidean metric maps a vector field $U = U^i \bfe_i$ to the one-form $U^\flat = U^i \bfe^i$. We also denote by $\del$ the {\em Euclidean connection}, i.e. the Levi-Civita connection of the Euclidean metric, which in the canonical basis reads $\del_U V = (U^i \partial_i) V^j \bfe_j$.

Over a set $\Omega\subset X$ where \eqref{eq:positivityn} holds, we have the following decomposition formula, true for any vector field $V$ over $\Omega$
\begin{align}
    V = H^{pq} (V\cdot S_p) S_q, \qquad H^{ij} = [H^{-1}]_{ij}.
    \label{eq:Vdecomp}
\end{align}
For any invertible symmetric matrix $M$, applying \eqref{eq:Vdecomp} to $MV$ and multiplying by $M^{-1}$ yields also the more general formula 
\begin{align}
    V = H^{pq} (V\cdot MS_p) M^{-1} S_q.
    \label{eq:MVdecomp}
\end{align}

\paragraph{Proof of equation \eqref{eq:gradlogtau}:}

The proof essentially relies on the study of the behavior of the {\em dual coframe}\footnote{For $(E_1,\cdots,E_n)$ a frame, $(\omega_1,\cdots,\omega_n)$ is called the dual coframe of $E$ if $\omega_i(E_j) = \delta_{ij}$ for $1\le i,j\le n$.} of the frame $(A^{-1}S_1,\cdots,A^{-1} S_n)$. Let us denote $\sigma_0 = \text{sgn}(\det (S_1,\dots,S_n))$, constant throughout $\Omega$ by virtue of \eqref{eq:positivityn}. Since $S^T S = H$ with $S = [S_1|\dots|S_n]$, we have that 
\begin{align}
    \det(S_1,\dots,S_n) = \sigma_0 \sqrt{\det H} = \sigma_0 \det H^{\frac{1}{2}}, \qquad H = \{H_{pq}\}_{1\le p,q\le n}. 
    \label{eq:detS}
\end{align}
For $1\le j\le n$, let us define the vector field $X_j$ by 
\begin{align}
    X_j^\flat = \sigma_j \star \left[ (A^{-1}S_1)^\flat \wedge\dots\wedge (A^{-1}S_{\hat j})^\flat\wedge\dots\wedge (A^{-1}S_n)^\flat  \right], \quad \sigma_j:= (-1)^{j-1}, 
    \label{eq:Xj}
\end{align}
where the hat over an index indicates its omission. $X_j$ is the unique vector field such that at every $x\in\Omega$ and for every vector $V\in T_x\Omega$, we have 
\begin{align*}
    X_j(x)\cdot V = \det (A^{-1} S_1,\dots, A^{-1} S_{j-1},\overbrace{V}^j, A^{-1} S_{j+1}, \dots, A^{-1} S_n ).
\end{align*} 
In particular, we have that for any $S_n^+(\Rm)$-valued function $M$ and any vector field $V$,
\begin{align*}
    MX_j\cdot V = X_j\cdot MV &= \det (A^{-1} S_1, \dots, MV, \dots, A^{-1}S_n) \\
    &= \det M \det ( (M^{-1} A^{-1})S_1, \dots, V, \dots,   (M^{-1} A^{-1})S_n),
\end{align*}
that is, we have that
\begin{align}
    (M X_j)^\flat = \sigma_j \det M \star\left[ (M^{-1} A^{-1}S_1)^\flat \wedge\dots\wedge (M^{-1} A^{-1}S_{\hat j})^\flat\wedge\dots\wedge (M^{-1}A^{-1}S_n)^\flat \right].
    \label{eq:MXj}
\end{align}
$(X_1,\cdots,X_n)$ is, up to some scalar factor, the dual basis to $(A^{-1} S_1,\cdots,A^{-1}S_n)$ since we have, for $i\ne j$
\begin{align*}
    X_j\cdot A^{-1} S_i = \det (A^{-1}S_1, \dots, \underbrace{A^{-1}S_i}_i  ,\dots, \underbrace{A^{-1}S_i}_j, \dots, A^{-1}S_n) = 0,
\end{align*} 
since the determinant contains twice the vector $A^{-1} S_i$. Moreover, when $i=j$, we have 
\begin{align*}
    X_j\cdot A^{-1}S_j = \det (A^{-1}S_1, \dots, A^{-1} S_n) = \det A^{-1} \det (S_1,\dots,S_n) = \sigma_0 \det \left( A^{-1} H^{\frac{1}{2}}  \right), 
\end{align*}
where we used relation \eqref{eq:detS}. Therefore we can use formula \eqref{eq:MVdecomp} with $M\equiv A^{-1}$ to represent $X_j$ as
\begin{align}
    X_j = \sum_{k,l=1}^n H^{kl} (X_j\cdot A^{-1} S_k) AS_l = \sigma_0 \sum_{l=1}^n H^{jl} \det (A^{-1} H^{\frac{1}{2}}) AS_l.
    \label{eq:Xjrep}
\end{align}

We now show that $X_j$ is divergence-free, that is $\nabla\cdot X_j = 0$ for $1\le j\le n$. Indeed, we write
\begin{align*}
    \nabla\cdot X_j = \star d\star X_j^\flat &= \star d \left[ (A^{-1}S_1)^\flat \wedge\dots\wedge (A^{-1}S_{\hat j})^\flat\wedge\dots\wedge (A^{-1}S_n)^\flat  \right] = 0,
\end{align*}
since an exterior product of closed forms is always closed, thus we have
\begin{align}
    \div X_j = 0, \quad 1\le j\le n. 
    \label{eq:divXj}
\end{align}
Combining equations \eqref{eq:Xjrep} together with \eqref{eq:divXj}, and using the identity $\div (fV) = f\div V + \nabla f\cdot V$ for $f$ a function and $V$ a vector field, we obtain
\begin{align*}
    0 = \div X_j &= \div (H^{jl} \det (A^{-1} H^{\frac{1}{2}}) AS_l) \\
    &= \det (A^{-1} H^{\frac{1}{2}}) \nabla H^{jl} \cdot A S_l + H^{jl} \nabla \det (A^{-1} H^{\frac{1}{2}})\cdot AS_l + \det (A^{-1} H^{\frac{1}{2}}) H^{jl} \div( AS_l).
\end{align*}
The last term is zero since $\div (AS_i) = 0$ and the second term expresses the dotproducts of $\nabla \det (A^{-1} H^{\frac{1}{2}})$ with the frame $A^{-1} S$. Thus we use the representation formula \eqref{eq:MVdecomp} with $M\equiv A$ and divide by $\det (A^{-1} H^{\frac{1}{2}})$ to obtain 
\begin{align*}
    \nabla \log \det (A^{-1} H^{\frac{1}{2}}) = H^{jl} (\nabla \log \det (A^{-1} H^{\frac{1}{2}})\cdot AS_l) A^{-1}S_j = - (\nabla H^{jl} \cdot A S_l) A^{-1} S_j,
\end{align*}
which upon writing $\log\det (A^{-1} H^{\frac{1}{2}}) = - \log\det A + \frac{1}{2} \log\det H$ yields
\begin{align*}
    \nabla\log\det A = \frac{1}{2} \nabla\log\det H + (\nabla H^{jl} \cdot A S_l) A^{-1} S_j.
\end{align*}
We now plug in the rescaling $A = \tau^{\frac{1}{2}} \wtA$ (so that $\det A = \tau^{\frac{n}{2}}$), which implies $A^{-1} = \tau^{-\frac{1}{2}} \wtA^{-1}$, and notice that the terms involving $\tau$ cancel out in the right-hand side of the last equation, thus \eqref{eq:gradlogtau} is proved. 

\paragraph{Proof of Equation \eqref{eq:gradSi}:}
We start by recalling the first equation of \eqref{eq:PDESi} 
\begin{align}
    d (\wtA^{-1}S_i)^\flat = F^\flat \wedge (\wtA^{-1}S_i)^\flat, \quad 1\le i\le n, \quad F:= \nabla\log\tau,
    \label{eq:PDESi2}
\end{align}
where $F$ is now considered a functional of $(S_1,\dots,S_n)$ and of the known power densities (this functional relation was obtained using the divergence equations in \eqref{eq:PDESi}, which are of no further use here). The main tool to derive a first-order differential system for $(S_1,\dots,S_n)$ from \eqref{eq:PDESi2} is Koszul's formula
\begin{align}
    2(\del_U V) \cdot W = \del_U (V\cdot W) + \del_V (U\cdot W) - \del_W (U\cdot V) - U\cdot [V,W] - V\cdot [U,W] + W\cdot[U,V], 
    \label{eq:Koszul}
\end{align}
which expresses covariant derivatives in terms of dotproducts and Lie-Brackets of given vector fields. The dotproducts of $S_1,\dots,S_n$ are known from the power densities, while the Lie Brackets $[\wtA S_i,\wtA S_j]$ are known from \eqref{eq:PDESi2}, as the following calculation shows 
\begin{align}
    \wtA^{-1} S_k\cdot [\wtA S_i,\wtA S_j] &= \wtA S_i\cdot \nabla H_{jk} - \wtA S_j\cdot\nabla H_{ik} - d (\wtA^{-1} S_k)^\flat (\wtA S_i, \wtA S_j) \nonumber\\
    &= \wtA S_i\cdot \nabla H_{jk} - \wtA S_j\cdot\nabla H_{ik} - F^\flat\wedge (\wtA^{-1}S_k)^\flat (\wtA S_i, \wtA S_j) \nonumber\\
    &= \wtA S_i\cdot \nabla H_{jk} - \wtA S_j\cdot\nabla H_{ik} - H_{kj} F\cdot\wtA S_i + H_{ki} F\cdot\wtA S_j, \label{eq:Lie}
\end{align}
where we have used \eqref{eq:PDESi2} and the characterization of the exterior derivative
\begin{align}
    dU^\flat (V,W) = \del_V (U\cdot W) - \del_W (U\cdot V) - U\cdot[V,W]. 
    \label{eq:dU}
\end{align}
Unless $\wtA = \Imm_n$, the frames $S$ and $\wtA S$ do not coincide, therefore one must modify formula \eqref{eq:Koszul} in order to obtain the promised system. 
Following \cite{Monard2012b}, we first choose to represent the total covariant derivative of $S_i$ in the basis of tensors of type $(1,1)$ given by $\{S_i \otimes (\wtA^{-1}S_j)^\flat \}_{i,j=1}^n$, in which the decomposition is explicitely given by 
\begin{align}
    \del S_i = H^{kq}H^{jp} (\del_{\wtA S_q} S_i\cdot S_p)\ S_j \otimes (\wtA^{-1} S_k)^\flat,
    \label{eq:gradSidecomp}
\end{align}
see \cite[Lemma 4.3.4]{Monard2012b}. The subsequent work consists in analysing the term $\del_{\wtA S_q} S_i\cdot S_p$, in particular removing all derivations on the $S_i$'s by moving them onto either known data $H_{ij}$ or the anisotropic structure $\wtA$. 

The first step is to establish the following ``modified'' Koszul formula
\begin{align}
    \begin{split}
	2 (\del_{\wtA S_q} S_i)\cdot S_p &= \del_{\wtA S_q} H_{ip} + \del_{\wtA S_i} H_{qp} - \del_{\wtA S_p} H_{qi} \\
	&\qquad - [S_i,S_p]^{\wtA}\cdot S_q - [S_q, S_p]^{\wtA}\cdot S_i + [S_q,S_i]^{\wtA}\cdot S_p,	
    \end{split}
    \label{eq:modifiedKoszul}
\end{align} 
where we have defined $[U,V]^\wtA = \del_{\wtA U} V - \del_{\wtA V} U$. Equation \eqref{eq:modifiedKoszul} is obtained in \cite[Lemma 4.3.2]{Monard2012b} by using the torsion-freeness and the compatibility of the connection with the metric.

The second step is to establish for any vector fields $U,V$ the following relation
\begin{align}
    [U,V]^\wtA = \wtA^{-1} [\wtA U, \wtA V] - \A_{\wtA} (U,V),
    \label{eq:AA}
\end{align}
where $\A_{\wtA}$ is a vector-valued tensor of type $(2,0)$, whose expression in local coordinates is expressed as 
\begin{align*}
    \A_{\wtA} (U,V) = \frac{1}{2} (U^l V^q - V^l U^q) \wtA^{-1} [\wtA_l,\wtA_q], \quad \wtA_l := \wtA(\cdot,\partial_l),
\end{align*} 
as established in \cite[Lemma 4.3.3]{Monard2012b}. Therefore, the tensor $\A_{\wtA}$ encodes differential information about the anisotropic structure and is identically zero if $\wtA$ is constant. Plugging relation \eqref{eq:AA} into \eqref{eq:modifiedKoszul}, and then using the known Lie brackets expression from \eqref{eq:Lie}, we arrive at
\begin{align*}
    2 (\del_{\wtA S_q} S_i)\cdot S_p &= \del_{\wtA S_q} H_{ip} + \del_{\wtA S_i} H_{qp} - \del_{\wtA S_p} H_{qi} \\
    &\qquad - [\wtA S_i,\wtA S_p] \cdot \wtA^{-1} S_q - [\wtA S_q, \wtA S_p]\cdot \wtA^{-1} S_i + [\wtA S_q,\wtA S_i]\cdot \wtA^{-1} S_p \\
    &\qquad + \A_{\wtA} (S_i,S_p) \cdot S_q + \A_{\wtA} (S_q,S_p)\cdot S_i - \A_{\wtA} (S_q,S_i) \cdot S_p \\
    &= \wtA S_q\cdot\nabla H_{ip} + \wtA S_p\cdot\nabla H_{iq} - \wtA S_i\cdot\nabla H_{pq} + 2 H_{pq} F\cdot\wtA S_i - 2 H_{qi} F\cdot\wtA S_p \\
    &\qquad + \A_{\wtA} (S_i,S_p) \cdot S_q + \A_{\wtA} (S_q,S_p)\cdot S_i - \A_{\wtA} (S_q,S_i) \cdot S_p.
\end{align*}
The last expression no longer differentiates the vector fields $S_i$, which fulfills our goal. Plugging the last expression into \eqref{eq:gradSidecomp} and simplifying expressions of the form \eqref{eq:Vdecomp}, we arrive at the expression
\begin{align}
    \begin{split}
	\del S_i &= \frac{1}{2} \left( S_k\otimes U_{ik}^\flat + \wtA U_{ik} \otimes (\wtA^{-1} S_k)^\flat + (\wtA S_i\cdot\nabla H^{jk})\ S_j\otimes (\wtA^{-1} S_k)^\flat \right)  \\
	&\qquad + (F\cdot \wtA S_i) \wtA^{-1} - \wtA F\otimes (\wtA^{-1} S_i)^\flat \\
	&\qquad + \frac{1}{2} H^{kq} H^{jp} (\A_\wtA (S_i,S_p) \cdot S_q + \A_\wtA (S_q,S_p) \cdot S_i - \A_\wtA (S_q,S_i)\cdot S_p)\ S_j \otimes (\wtA^{-1} S_k)^\flat,
    \end{split}
    \label{eq:gradSifinal}
\end{align}
where we have defined the data vector fields 
\begin{align}
    U_{jk} := (\nabla H_{jp}) H^{pk} = - H_{jp} (\nabla H^{pk}), \quad 1\le j,k\le n.
    \label{eq:Ujk}
\end{align}
The last thing to notice is that, with the expression of $F$ given by \eqref{eq:gradlogtau}, the right-hand side of \eqref{eq:gradSifinal} is a polynomial in the components of $S$ of order at most five. Together with the {\it a priori} uniform estimate
\begin{align*}
    \sum_{i=1}^n |S_i(x)|^2 \le n\max_{x\in X,\ 1\le i\le n} H_{ii}(x),
\end{align*}
we deduce that $\del S_i = \S_i (S,H,dH,\wtA,d\wtA)$, where $\S_i$ is Lipschitz-continuous with respect to $S$ so that the method of characteristic is a uniquely defined and stable. 

\subsection{Reconstruction of the anisotropic structure $\tilde\gamma$, then of $\tau$} \label{sec:tildegamma}

As explained in Section \ref{sec:statement}, we start with $n$ solutions $(u_1,\cdots,u_n)$ whose gradients satisfy the rank maximality condition \eqref{eq:positivityn} over some $\Omega\subset X$. We will call $(\nabla u_1,\cdots,\nabla u_n)$ the {\em support basis}. 

\paragraph{Algebraic redundancies:} 

Using the support basis above and formula \eqref{eq:Vdecomp}, we have for any additional solution $u_{(i)}$ the following relation
\begin{align*}
    \nabla u_{(i)} = H^{pq} (\nabla u_{(i)}\cdot \gamma\nabla u_p) \nabla u_q = H^{pq} H_{(i)p} \nabla u_q.
\end{align*}
As a result, the power density of any two additional solutions $u_{(i)}$ and $u_{(j)}$ is computible via the formula 
\begin{align*}
    H_{(i)(j)} = \nabla u_{(i)}\cdot\gamma\nabla u_{(j)} = H^{pq} H_{(i)p} H^{rl} H_{(j)r} \nabla u_q\cdot\gamma\nabla u_l = H^{pr} H_{(i)p}H_{(j)r},
\end{align*}
i.e. the mutual power density of any two additional solutions is algebraically computible from the mutual power densities of each of these solutions with the support basis. In other words, any additional solution $u_{(i)}$ adds at most $n$ non-redundant dimensions of data, that is, the quantities $H_{(i)p}$ for $1\le p\le n$.

\paragraph{Algebraic equations for $\wtA S$:}

Let us add an additional solution $v \equiv u_{n+1}$ and consider the mutual power densities of these $n+1$ solutions $\{H_{ij}\}_{1\le i,j\le n+1}$. As explained in Appendix \ref{app:lindep}, by linear dependence of $n+1$ vectors in $\Rm^n$, one can find $n+1$ functions 
\begin{align}
    \begin{split}
	\mu_i &= (-1)^{i+n+1} \det \{H_{pq}\}_{1\le p\le n,\ 1\le q\le n+1,\ q\ne i}, \quad 1\le i\le n, \\
	\mu  &= \det \{H_{pq}\}_{1\le p,q\le n},	
    \end{split}    
    \label{eq:mu}
\end{align}
such that 
\begin{align}
    \sum_{i=1}^{n} \mu_i S_i + \mu A\nabla v = 0,
    \label{eq:lindepS}
\end{align}
where $\mu$ never vanishes over $\Omega$ by virtue of \eqref{eq:positivityn}. In particular, the following expression is well-defined over $\Omega$
\begin{align}
    S_v = -\mu^{-1} \sum_{i=1}^n \mu_i S_i.
    \label{eq:Sv}
\end{align}
We now apply the operators $d(A^{-1}\cdot)^\flat$ and $\nabla\cdot(A\cdot)$ to relation \eqref{eq:Sv}, and using the fact that $d(A^{-1} S_i)^\flat = 0$ and $\nabla\cdot(A S_i) = 0$, we arrive at the following relations
\begin{align}
    Z_i^\flat \wedge (\wtA^{-1}S_i) = 0 \qandq Z_i\cdot\wtA S_i = 0, \where \quad Z_i:= \nabla \frac{\mu_i}{\mu}. 
    \label{eq:orthogonalityA}
\end{align}
The first equation describes the vanishing of a two-form, which amounts to $n(n-1)/2$ scalar relations, obtained by applying this two-form to vector fields $\wtA S_p, \wtA S_q$ for $1\le p<q\le n$:
\begin{align*}
    H_{iq} Z_i\cdot\wtA S_p - H_{ip} Z_i\cdot \wtA S_q = 0, \quad 1\le p<q\le n.
\end{align*}
Put in other terms and defining $Z:= [Z_1|\dots|Z_n]$ and $S:= [S_1|\dots|S_n]$, these relations express the facts that {\em the matrix $Z^T\wtA S$ is traceless}, and that {\em the matrix $HZ^T \wtA S$ is symmetric}, which we may express as orthogonality statements of the form 
\begin{align}
    \dprod{\wtA S}{Z} = 0 \qandq \dprod{\wtA S}{Z H \Omega} = 0, \quad \Omega\in A_n(\Rm).
    \label{eq:orthogonality}
\end{align}
In other words, the matrix $\wtA S$ is orthogonal to the following subspace of $M_n(\Rm)$ 
\begin{align}
    \V := \{ Z (\lambda \Imm_n + H\Omega),\quad (\lambda,\Omega) \in \Rm\times A_n(\Rm) \}.
    \label{eq:V}
\end{align}
As established in \cite[Proposition 4.3.8]{Monard2012b}, we have that $\dim V = 1 + r(n - (r+1)/2)$, where $r = \rank Z$, with maximal value $1 + n(n-1)/2$ when $r\in \{n-1,n\}$. 

\paragraph{Reconstruction algorithm:} Assume now that we have $l\ge 1$ additional solutions $(v_1,\dots,v_l)$ generating spaces $\V_1,\dots,\V_l$ of the form \eqref{eq:V}. Let $\{ \bfe_p\otimes \bfe^q - \bfe_q\otimes \bfe^p, \quad 1\le p<q\le n\}$ be a basis for $A_n(\Rm)$, then the space $\sum_{i=1}^l \V_i$ is spanned by the following family
\begin{align}
    \M = \{Z_i,\ Z_i H (\bfe_p\otimes \bfe^q - \bfe_q\otimes \bfe^p)\ |\ 1\le i\le l, \quad 1\le p<q\le n \},\quad \#\M = d_M l, 
    \label{eq:M}
\end{align}
with $d_M$ defined in \eqref{eq:dm}. Assuming that $\dim \sum_{i=1}^l \V_i = n^2-1$ throughout $\Omega$, there is a $n^2-1$-family of $\M$ spanning it, from which we would like to reconstruct $\wtA S$ via a cross product formula, see Appendix \ref{app:crossprod}. Now, for any $n^2-1$-subfamily $(M_1, \dots, M_{n^2-1})$ of $\M$, the cross product $\N(M_1,\dots,M_{n^2-1})$ is
\begin{itemize}
    \item[(i)] either zero if $(M_1, \dots, M_{n^2-1})$ is linearly dependent,
    \item[(ii)] equal to $\pm \left| \frac{\det \N(M_1,\dots,M_{n^2-1})}{\det (\wtA S)} \right|^\frac{1}{n} \wtA S$ otherwise. 
\end{itemize}
In the second case, we compute
\begin{align*}
    \N(M_1,\dots,M_{n^2-1}) H^{-1} \N^2(M_1,\dots,M_{n^2-1}) = \left| \frac{\det \N(M_1,\dots,M_{n^2-1})}{\det (\wtA S)} \right|^\frac{2}{n} \wtA S H^{-1} S^T \wtA^T.
\end{align*}
Using the fact that $S H^{-1} S^T = \Imm_n$, $\wtA = \wtA^T$, and $(\det (\wtA S))^2 = \det H$, we deduce the relation 
\begin{align*}
    \N H^{-1} \N^T = (\det (\N H^{-1} \N^T))^{\frac{1}{n}} \tilde\gamma, \quad\text{with}\quad \N:= \N(M_1,\dots,M_{n^2-1}).
\end{align*}
This expression also covers the case (i), as the factor in front of $\tilde \gamma$ is zero if $(M_1, \dots, M_{n^2-1})$ is linearly dependent. As we do not know {\it a priori} which subfamily of $\M$ is independent, we may sum over all possible cases. With the notation $\I(n^2-1,d_M l)$ introduced in Section \ref{sec:statement}, we sum the last reconstruction formula over all possible $n^2-1$-subfamilies of $\M$ 
\begin{align}
    \sum_{I\in\I(n^2-1, d_M l)} \!\!\!\!\!\! \N(I) H^{-1} \N (I)^T = \F\ \tilde \gamma, \where \quad \F := \sum_{I\in\I(n^2-1,d_M l)} \!\!\!\!\!\! (\det (\N(I) H^{-1} \N (I)^T))^{\frac{1}{n}}.
    \label{eq:F}
\end{align}
$\F$ is a sum of nonnegative terms which vanishes precisely when $\dim \sum_{i=1}^l \V_i < n^2-1$, so a way of formulating the fact that $\dim \sum_{i=1}^l \V_i = n^2-1$ and that $\left( \sum_{i=1}^l \V_i \right)^\perp$ is spanned by a nonsingular matrix is indeed
\begin{align}
    \inf_{x\in \Omega} \F(x) \ge c_1>0.
    \label{eq:infF}
\end{align}
When condition \eqref{eq:infF} is satisfied, $\tilde\gamma$ is uniformly and uniquely reconstructed over $\Omega$ by the following formula
\begin{align}
    \tilde\gamma = \frac{1}{\F} \sum_{I\in\I(n^2-1,d_M l)} \N(I) H^{-1} \N (I)^T, \quad x\in \Omega.
    \label{eq:gammarecons}
\end{align}
On to the reconstruction of $\tau$, we restart from \eqref{eq:gradlogtau} and rewrite it as 
\begin{align}
    \tilde\gamma\nabla\log\tau = \frac{2}{n} |H|^{-\frac{1}{2}} \left( \nabla (|H|^{\frac{1}{2}}H^{jl})\cdot\wtA S_l \right)\wtA S_j,
    \label{eq:gradlogtau2}
\end{align}
where $\tilde\gamma$ is obtained from \eqref{eq:gammarecons}. Again, we will use the cross-product expression to express $\nabla\log\tau$ solely in terms of data. For $I\in\I(n^2-1, d_M l)$, we have,
\begin{align*}
    \N(I) = \pm \left| \frac{\det \N(I)}{\sqrt{\det H}} \right|^{\frac{1}{n}} \wtA S,
\end{align*}
where there is a sign indeterminacy. However, this indeterminacy disappears when considering expressions as in the right-hand side of \eqref{eq:gradlogtau2}:

\begin{align*}
    \left( \nabla (|H|^{\frac{1}{2}}H^{jl})\cdot \N(I)\bfe_l \right)\N(I)\bfe_j &= (\pm)^2 \left| \frac{\det \N(I)}{\sqrt{\det H}} \right|^{\frac{2}{n}} \left( \nabla (|H|^{\frac{1}{2}}H^{jl})\cdot\wtA S_l \right)\wtA S_j \\
    &= ( \det (\N(I)H^{-1}\N(I)^T )^{\frac{1}{n}} \left( \nabla (|H|^{\frac{1}{2}}H^{jl})\cdot\wtA S_l \right)\wtA S_j.
\end{align*}
Summing over $I\in\I(n^2-1,d_M l)$, we arrive at
\begin{align*}
    \sum_{I\in\I} \left( \nabla (|H|^{\frac{1}{2}}H^{jl})\cdot \N(I)\bfe_l \right)\N(I)\bfe_j = \F \left( \nabla (|H|^{\frac{1}{2}}H^{jl})\cdot\wtA S_l \right)\wtA S_j = \F |H|^{\frac{1}{2}} \frac{n}{2} \tilde\gamma\nabla\log\tau,
\end{align*}
which finally may be inverted as 
\begin{align}
    \nabla\log\tau = \frac{2}{n \F |H|^{\frac{1}{2}}} \sum_{I\in\I} \left( \nabla (|H|^{\frac{1}{2}}H^{jl})\cdot \N(I)\bfe_l \right) \tilde\gamma^{-1}\N(I)\bfe_j, \quad x\in \Omega.
    \label{eq:gradlogtaurecons}
\end{align}
This reconstruction formula guarantees a unique and stable reconstruction of $\tau$ with no ambiguity. 

\begin{proof}[Proof of Proposition \ref{prop:stabgamma}]
    The reconstruction of $(\tilde\gamma,\tau)$ is based on formulas \eqref{eq:gammarecons} and \eqref{eq:gradlogtaurecons}. Putting definitions \eqref{eq:mu}, \eqref{eq:V} and \eqref{eq:crossprod} together, we see that the right-hand side of \eqref{eq:gammarecons} is, at every point, a polynomial of power densities and their first-order derivatives. Since the denominator $\F$ in \eqref{eq:gammarecons} is bounded away from zero, we clearly have a continuity statement of the form 
    \begin{align*}
	\|\tilde\gamma - \tilde\gamma'\|_{L^\infty(\Omega)} \le C \|H-H'\|_{W^{1,\infty}(X)},
    \end{align*}
    where the constant $C$ degrades like $c_1^{-1}$ with $c_1$ the constant in \eqref{eq:infF}. On to the reconstruction of $\log\tau$, we can make the same observation as before judging by equation \eqref{eq:gradlogtaurecons} and the fact that, since $\det\tilde\gamma = 1$, the entries of $\tilde\gamma^{-1}$ are polynomials in the entries of $\tilde\gamma$. This leads to a stability statement of the form
    \begin{align*}
	\|\nabla (\log \tau - \log\tau')\|_{L^\infty(\Omega)} \le C \|H-H'\|_{W^{1,\infty}(X)},
    \end{align*}    
    where $C$ here degrades like $c_1^{-2}$. Proposition \ref{prop:stabgamma} is proved.
\end{proof}


\subsection{Proof of Theorem \ref{thm:local}} \label{sec:prooflocal}

The proof of Theorem \ref{thm:local} uses the Runge approximation for elliptic equations, which by virtue of \cite[Equivalence Theorem p.442]{Lax1956} is equivalent to the unique continuation property. The latter property holds for conductivity tensors with regularity no lower than Lipschitz \cite{Garofalo1986}. The Runge approximation, as it is proved in \cite{Lax1956,Bal2012} for instance and adapted to our case here, states that if $\Omega\subset\subset X$, then any function $u\in H^1(\Omega)$ satisfying $\nabla\cdot(\gamma\nabla u)=0$ over $\Omega$ can be approximated arbitrarily well in the sense of $L^2(\Omega)$ by solutions of \eqref{eq:conductivity}, provided that $\gamma$ is Lipschitz-continuous throughout $X$. In fact, we require a little more regularity here ($\gamma\in\C^{1,\alpha}(X)$ with $0<\alpha<1$) for forward elliptic estimates.

\paragraph{Step 1. Local solutions with constant coefficients:} 
Fix $x_0\in X$ and $B_{3r} \equiv B_{3r}(x_0)\subset X$ a ball of radius $3r$ ($r$ tuned hereafter) centered at $x_0$. Denote $\gamma_0:= \gamma(x_0)$ and $A_0 := \gamma_0^{\frac{1}{2}}$. We first construct solutions to the problem with constant coefficients, whose power densities will satisfy conditions \eqref{eq:positivityn} and \eqref{eq:hyperplane}. Such solutions are given by 
\begin{align}
    \begin{split}
	u^0_i(x) &:= x^i-x_0^i, \qquad 1\le i\le n, \qquad\text{and for} \quad 1\le j\le n-1, \\ 	
	u_{n+j}^0(x) &:= \frac{1}{2} (x-x_0)\cdot Q_j (x-x_0), \qquad Q_j:= A_0^{-1} \Hm_j A_0^{-1},
    \end{split}
    \label{eq:constcoeffs}
\end{align}
where we have defined $\Hm_j:= \bfe_j\otimes\bfe^j - \bfe_{j+1}\otimes \bfe^{j+1}$. These solutions satisfy $\nabla\cdot(\gamma_0 \nabla u) = 0$ throughout $\Rm^n$, and we trivially have 
\begin{align}
    \det (\nabla u_1^0,\dots, \nabla u_n^0) = 1, \quad x\in \Rm^n,
    \label{eq:detconstant}
\end{align}
so that condition \eqref{eq:positivityn} is satisfied throughout $B_{3r}$. Moreover, condition \eqref{eq:hyperplane} is also satisfied as direct calculations lead to $Z_{i} = Q_i = A_0^{-1} \Hm_i A_0^{-1}$ for $1\le i\le n-1$, and the matrix $H^0:= \{\nabla u_i^0\cdot\gamma_0\nabla u_j^0\}_{i,j=1}^n$ is nothing but $\gamma_0$. Thus the space of orthogonality is given by
\begin{align*}
    \V = \sum_{i=1}^l \Rm Q_i + Q_i\ \gamma_0\ A_n(\Rm) = A_0^{-1} \left(\sum_{i=1}^l \Rm \Hm_i + \Hm_i A_n(\Rm) \right) A_0^{-1}.
\end{align*}
The last space between brackets can easily be seen to not depend on $x$ and it spans the hyperplane of traceless matrices $\{\Imm_n \}^\perp$, so that $\V = \{\gamma_0\}^\perp$. In particular, condition \eqref{eq:hyperplane} is satisfied for some constant $c_1 >0$ independent of $x$.

\paragraph{Step 2. Local solutions with varying coefficients:}
From solutions $\{u_i^0\}_{i=1}^{2n-1}$, we construct a second family of solutions $\{u_i^r\}_{i=1}^{2n-1}$ via the following equation
\begin{align}
    \nabla\cdot(\gamma \nabla u_i^r) = 0\quad (B_{3r}), \quad u_i^r|_{\partial B_{3r}} = u_i^0, \qquad 1\le i\le 2n-1,
    \label{eq:localsol}
\end{align}
thus the maximum principle implies that 
\begin{align}
    \max_{1\le i\le n} \|u_i^r\|_{L^\infty(B_{3r})} \le 3r \qandq \max_{1\le j\le n-1} \|u_{n+j}^r\|_{L^\infty (B_{3r})} \le C r^2,
    \label{eq:maxprinc}
\end{align}
where the constant only depends on the constant of ellipticity $C(\gamma)$. The difference of both solutions satisfies, for $1\le i\le 2n-1$,
\begin{align}
    - \nabla\cdot(\gamma\nabla(u_i^r - u_i^0)) = \nabla\cdot ((\gamma-\gamma_0)\nabla u_i^0) \quad (B_{3r}), \quad (u_i^r-u_i^0)|_{\partial B_{3r}} = 0,
    \label{eq:localdiff}
\end{align}
where the right-hand side belongs to $\C^{0,\alpha}(\overline{B_{3r}})$ with a uniform bound in $0\le r\le r_0$ for some $r_0$. 
Thus \cite[Theorem 6.6]{Gilbarg2001} implies that 
\begin{align}
    \|u_i^r-u_i^0\|_{\C^{2,\alpha}(B_{3r})} \le C( \|u_i^r-u_i^0 \|_{L^\infty (B_{3r})} + \|F_i\|_{\C^{1,\alpha}(B_{3r})}) \le C' \|\gamma\|_{\C^{1,\alpha}(X)}, 
    \label{eq:holder2}
\end{align}
where the first constant depends on $n$, $C(\gamma)$, $\|\gamma\|_{\C^{1,\alpha}}$ and $B_{3r}$. Interpolating between \eqref{eq:maxprinc} and \eqref{eq:holder2}, we deduce the first important fact
\begin{align}
    \lim_{r\to 0} \quad  \max_{1\le i\le 2n-1} \|u_i^r - u_i^0\|_{\C^2(B_{3r})} = 0.
    \label{eq:localestimate}
\end{align}

\begin{remark}[Dependency of the constants on the domain]
    The constant in \eqref{eq:holder2} depends on $\partial B_{3r}$, thus on $r$, however this dependency works in our favor when shrinking the domain. This can be seen by rescaling the problem $x\to x_0 + rx', x'\in B_3(0)$ to keep the domain fixed, and studying the behavior of the constants w.r.t. the rescalings. 
\end{remark}

\paragraph{Step 3. Runge approximation (control from the boundary $\partial X$):} Assume $r$ has been fixed at this stage. By virtue of the Runge approximation property, for every $\varepsilon >0$ and $1\le i\le 2n-1$, there exists $g_i^\varepsilon\in H^{\frac{1}{2}} (\partial X)$ such that 
\begin{align}
    \|u_i^\varepsilon - u_i^r\|_{L^2(B_{3r})} \le \varepsilon, \where\ u_i^\varepsilon \text{ solves } \eqref{eq:conductivity} \text{ with } u_i^\varepsilon|_{\partial X} = g_i^\varepsilon.
    \label{eq:L2est}
\end{align}
Now applying \cite[Theorem 8.24]{Gilbarg2001} using the fact that $\nabla\cdot(\gamma\nabla (u_i^\varepsilon-u_i^r))=0$ thoughout $B_{3r}$, we deduce that there exists $\beta>0$ such that
\begin{align}
    \|u_i^\varepsilon - u_i^r\|_{\C^\beta(\overline{B_{2r}})} \le C\|u_i^\varepsilon - u_i^r\|_{L^2(B_{3r})} \le C\varepsilon,
    \label{eq:Calphaest}
\end{align}
where the constant only depends on $n$, $C(\gamma)$ and $r = \dist (B_{2r},\partial B_{3r})$, in particular the same estimate holds with $\|u_i^\varepsilon - u_i^r\|_{L^\infty(B_{2r})}$ on the left-hand side. Finally, combining \eqref{eq:Calphaest} with \cite[Corollary 6.3]{Gilbarg2001}, we arrive at 
\begin{align*}
    \|u_i^\varepsilon - u_i^r\|_{\C^2 (\overline{B_r})} \le \frac{C}{r^2} \|u_i^\varepsilon - u_i^r\|_{L^\infty(B_{2r})} \le \frac{C}{r^2} \varepsilon,
\end{align*}
where the constant only depends on $\alpha$, $n$, $C(\gamma)$ and $\|\gamma\|_{\C^{1,\alpha}(X)}$. Since $r$ is fixed at this stage, we deduce that 
\begin{align}
    \lim_{\varepsilon\to 0} \quad \max_{1\le i\le 2n-1} \|u_i^\varepsilon-u_i^l\|_{\C^2(B_r)} = 0.
    \label{eq:epsilonestimate}
\end{align}

\paragraph{Completion of the argument:}
For any $\Omega\subset X$, the following functionals are continuous
\begin{align*}
    C^{1,\alpha}(\Omega)\times C^2(\Omega)\times C^2(\Omega) \ni (\gamma,u,v) &\mapsto H(\gamma,u,v) = \nabla u\cdot\gamma\nabla v \in W^{1,\infty}(\Omega), \\
    [W^{1,\infty} (\Omega)]^{n(n+1)/2} \ni \{H_{ij}\}_{1\le i\le j\le n} &\mapsto \det \{H_{ij}\}_{i,j=1}^n \in W^{1,\infty}(\Omega), \\
    H = \{H_{ij}\}_{1\le i\le j\le 2n-1} &\mapsto \F(H, \nabla H) \in L^\infty(\Omega),
\end{align*}
where in the last case, $\F$ is defined in \eqref{eq:F} with $l=n-1$ and its the domain of definition is 
\begin{align*}
    [W^{1,\infty} (\Omega)]^{(2n-1)n} \quad \text{with the condition}\quad \inf_{\Omega} \det \{H_{ij}\}_{i,j=1}^n >0.
\end{align*}
Step 1 established that $\det \{H_{ij}^0\}_{1\le i\le j\le n}$ and $\F(H^0,\nabla H^0)$ were bounded away from zero over $B_r$. Due to the limits \eqref{eq:localestimate} and \eqref{eq:epsilonestimate}, there exists a small $r>0$, then a small $\varepsilon>0$ such that $\max_{1\le i\le 2n-1} \|u_i^\varepsilon-u_i^0\|_{\C^2(B_r(x_0))}$ is so small that, by the continuity of the functionals mentioned above, $\det \{H_{ij}^\varepsilon\}_{1\le i\le j\le n}$ and $\F(H^\varepsilon,\nabla H^\varepsilon)$ remain uniformly bounded from zero over $B_r$, where we have denoted $H^\varepsilon_{ij} := \nabla u_i^\varepsilon\cdot\gamma\nabla u_j^\varepsilon$ for $1\le i,j\le 2n-1$. Conditions \eqref{eq:positivityn} and \eqref{eq:hyperplane} are thus satisfied over $B_r$ by the family $\{u_i^\varepsilon\}_{i=1}^{2n-1}$ which is controlled by boundary conditions. The proof of Theorem \ref{thm:local} is complete.

\section{Global questions}\label{sec:global}

\subsection{Admissibility sets and their properties}\label{sec:admset}

For compactness of notation, we denote by $\I(M,N)$ ($M\le N$) the set of increasing injections from $[1,M]$ to $[1,N]$ (i.e. $I\in\I(M,N)$ has the form $I = (i_1,\dots,i_M)$ with $1\le i_1<\dots<i_M\le N$).

\paragraph{The sets $\G_\gamma$:} The first admissibility set is that of boundary conditions ensuring that the scalar factor $\tau$ is uniquely and stably reconstructible. This requires the existence of, locally, $n$ solutions with linearly independent gradients. Although one can easily choose $m=n$ in two dimensions thanks to \cite[Theorem 4]{Alessandrini2001}, some counterexamples in higher dimensions \cite{Laugesen1996,Briane2004} show that one may need stricly more than $n$ solutions in general, hence the definition below.

\begin{definition}[Admissibility set $\G_\gamma^m,\ m\ge n$]\label{def:Ggammam}
    Let $\gamma \in \Sigma(X)$ be a given conductivity tensor. For $m\ge n$, an $m$-tuple $\bfg = (g_1,..,g_m)\in (H^{\frac{1}{2}}(\partial X))^m$ belongs to $\G_\gamma^m$ if the following conditions are satisfied (denote $u_i$ the solution of \eqref{eq:conductivity} with $u_i|_{\partial X} = g_i$):
    \begin{itemize}
	\item[(i)] The power densities $H_{ij} = \nabla u_i\cdot\gamma\nabla u_j$ belong to $W^{1,\infty}(X)$ for $1\le i,j\le m$.
	\item[(ii)] There exists a constant $C_{\bfg}>0$ such that
	    \begin{align}
		\inf_{x\in X} \D_\gamma^m [\bfg] (x) \ge C_{\bfg}, \where \quad \D^m_\gamma[\bfg](x) := \sum_{I\in\I(n,m)} \det \{ H_{pq} \}_{p,q\in I}.
		\label{eq:positivitym}
	    \end{align}
    \end{itemize}
\end{definition}

Condition {\em (i)} above allows to construct a finite open cover of $X$ in a generic manner, where to each open set $\Omega_k$ can be associated a single basis of $n$ solutions, see \cite[Prop. 5.1.2]{Monard2012b}. This basis can then be used to reconstruct $\tau$ throughout each $\Omega_k$. Doing this for each $\Omega_k$ and patching reconstructions appropriately allows to reconstruct $\tau$ in a globally unique and stable fashion, as is summarized in \cite[Theorem 5.1.4]{Monard2012b}.

\paragraph{The sets $\A_\gamma$:}

On to the global reconstruction of the anisotropy $\tilde\gamma$, we now define a second class of sets of boundary conditions, such that the solutions generated satisfy condition \eqref{eq:hyperplane} throughout $X$.

Let $\gamma$ such that $\G_\gamma^m\ne\emptyset$ for some $m\ge n$ and pick $\bfg\in\G_\gamma^m$ with constant $C_{\bfg}$ as in \eqref{eq:positivitym}. By virtue of \cite[Prop. 5.1.2]{Monard2012b}, there exists an open cover made of balls $\O = \{\Omega_k\}_{k=1}^K$ of $X$, a constant $C'_\bfg>0$ and an indexing function $I_{(k)} = (i_{(k)1},\dots,i_{(k)n})\in\I(n,m)$ for $1\le k\le K$ 
\begin{align}
    \min_{1\le k\le K} \inf_{x\in\Omega_k} \det H_{(k)} \ge C'_{\bfg}, \qquad H_{(k)} := \{H_{pq},\ p,q\in I_{(k)} \},
    \label{eq:opencover}
\end{align}
i.e. one may use $\{\nabla u_i\}_{i\in I_{(k)}}$ as a support basis over $\Omega_k$. Given an additional solution $v_{\alpha}$, we now construct over each $\Omega_k$ a basis for the space $\V$ based on the local support basis:
\begin{align}
    \begin{split}
	\V_\alpha|_{\Omega_k} &= \Rm Z_{\alpha(k)} + Z_{\alpha(k)} H_{(k)} A_n(\Rm), \quad \text{where for } 1\le j\le n, \\
	Z_{\alpha(k)}\bfe_j &:= (-1)^{j+n+1} \nabla \left( \det \{H_{pq},\ p\in I_{(k)}, q\in\text{subs }(I_{(k)}, j,\alpha) \}\ /\ \det H_{(k)} \right)
    \end{split}    
    \label{eq:Vomegak}
\end{align}
and where ``$\text{subs }(I_{(k)}, i_{(k)j},\alpha)$'' is obtained from $I_{(k)}$ by replacing $i_{(k)j}$ by $\alpha$. From a collection of $l$ additional solutions, similarly to \eqref{eq:M}, we build over each $\Omega_k$ the family of matrices
\begin{align}
    \M|_{\Omega_k} = \{ Z_{i(k)},\ Z_{i(k)} H_{(k)} (\bfe_p\otimes \bfe^q - \bfe_q\otimes \bfe^p)\ |\ \quad 1\le i\le l,\quad 1\le p<q\le n \},
    \label{eq:Mp}
\end{align}
of cardinality $d_M l$ with $d_M$ defined in \eqref{eq:dm}, so that we may rewrite it generically as 
\begin{align*}
    \M|_{\Omega_k} = \{ M_{(k)i}\ |\ 1\le i\le d_M l \}.    
\end{align*}

\begin{definition}[Admissibility set $\A_\gamma^{m,l}(\bfg)$ for $\bfg\in\G_\gamma^m$]\label{def:Agammamk} 
    For $m\ge n$, let us assume that $\bfg = (g_1,\cdots,g_m) \in \G_\gamma^m$, and let $(\O = \{\Omega_k \}_{k=1}^K, I, C_\bfg)$ an open cover, an indexing function and a constant associated to it. For $l\ge 1$, we say that $l$ additional boundary conditions $\bfh = (h_1,\cdots,h_l) \in \big(H^{\frac{1}{2}}(\partial X)\big)^l$ belong to the set of admissibility $\A_\gamma^{m,l}(\bfg)$ if there exists a constant $C_{\bfg,\bfh}>0$ such that the following condition holds
    \begin{align}
	\min_{1\le k\le K} \inf_{x\in \Omega_k} & \F_\gamma^{m,l} [\bfg,\bfh]|_{\Omega_k} (x) \ge C_{\bfg,\bfh}, \where \label{eq:hyperplane_det} \\
	\F_\gamma^{m,l}[\bfg,\bfh]|_{\Omega_k} &:= \sum_{J\in\I(n^2-1,d_M l)} \det (\N_{(k)} (J) H_{(k)}^{-1} \N_{(k)} (J)^T)^{\frac{1}{n}}, \label{eq:Fgammaml} \\
	\N_{(k)} (J) &:= \N (M_{(k)j_1}, \dots, M_{(k)j_{n^2-1}}). \nonumber
    \end{align}
\end{definition}

With definitions \ref{def:Ggammam} and \ref{def:Agammamk} in mind, in the sense of the present derivations, we may say that a tensor $\gamma$ is {\em globally reconstructible from power densities} if $\G_\gamma^m\ne\emptyset$ for some $m\ge n$ and for $\bfg\in\G_\gamma^m$, $\A_\gamma^{m,l}(\bfg)\ne\emptyset$ for $l\ge 1$ large enough.

\subsection{Properties of the admissibility sets}\label{sec:admsetprop}

\paragraph{Openness properties of $\G_\gamma$ and $\A_\gamma$:}

\begin{itemize}
    \item For $\C^{1,\alpha}$-smooth $\gamma$ an $\C^3$-smooth $\partial X$, the sets $\G_\gamma$ and $\A_\gamma$ are open for the topology of $\C^{2,\alpha}(\partial X)$ boundary conditions (\cite[Lemma 5.2.2]{Monard2012b}). 
    \item For $\C^{1,\alpha}$-smooth $\gamma$ (\cite[Lemma 5.2.3]{Monard2012b}).
\end{itemize}

\paragraph{Behavior of $\G_\gamma$ and $\A_\gamma$ with respect to push-forwards by diffeomorphisms:}

In the topic of inverse conductivity, diffeomorphisms are used in the anisotropic Calder\'on's problem to exhibit an obstruction to uniqueness. Here, these diffeomorphisms work in our favor in the sense that the property of being locally or globally reconstructible from power densities carries through push-forwards by diffeormorphisms. 

Let $\Psi:X\to \Psi(X)$ be a $W^{1,2}$-diffeomorphism where $X$ has smooth boundary. Then for $\gamma\in\Sigma(X)$, we define over $\Psi(X)$ the push-forward of $\gamma$ by $\Psi$, denoted $\Psi_\star\gamma$, the tensor
\begin{align}
    \Psi_\star \gamma (y) := (|J_\Psi|^{-1} D\Psi\ \gamma\ D\Psi^T)\circ \Psi^{-1}(y), \quad y\in\Psi(X), \quad J_\Psi:= \det D\Psi.
    \label{eq:pushfwd}
\end{align}
As explained in \cite{Astala2005}, $\Psi_\star \gamma \in \Sigma(\Psi(X))$, and $\Psi$ pushes foward a solution $u$ of \eqref{eq:conductivity} to a function $v = u\circ\Psi^{-1}$ satisfying the elliptic equation
\begin{align*}
    -\nabla_y \cdot (\Psi_\star \gamma \nabla_y v) = 0 \quad (\Psi(X)), \quad v|_{\partial (\Psi(X))} = g\circ\Psi^{-1},
\end{align*}
moreover $\Psi$ and $\Psi|_{\partial X}$ induce isomorphisms of $H^1(X)$ and $H^{\frac{1}{2}}(\partial X)$ onto $H^1(\Psi(X))$ and $H^{\frac{1}{2}}(\partial (\Psi(X)))$, respectively. For our proofs based on pointwise estimates, we will add the further requirement that $\Psi$ satisfies a condition of the form 
\begin{align}
    C_\Psi^{-1} \le |J\Psi(x)| \le C_\Psi, \quad x\in X, \quad \text{for some constant } C_\Psi \ge 1.
    \label{eq:psiOK}
\end{align}
We define the relation $(\gamma,X)\sim (\gamma',X')$ iff there exists $\Psi:X\to X'$ a diffeomorphism onto $X'$ satisfying \eqref{eq:psiOK}, such that $\gamma' = \Psi_\star \gamma$. It is clear that $\sim$ is an equivalence relation. 

With these definitions in mind, our main observation is the following
\begin{proposition}[Prop. 5.2.4-5.2.5 in \cite{Monard2012b}]\label{prop:pushfwd}
    For $\gamma\in\Sigma(X)$ and $\Psi:X\to\Psi(X)$ a $W^{1,2}$-diffeomorphism satisfying \eqref{eq:psiOK}, we have for any $m\ge n$ 
    \begin{align}
	\G_{\Psi_\star\gamma}^m = \{\bfg\circ\Psi^{-1}\ :\ \G_\gamma^m \}.
	\label{eq:pushfwdGgamma}
    \end{align}
    Moreover, if $\bfg\in\G_\gamma^m$ for some $m\ge n$, then we have 
    \begin{align}
	\A_{\Psi_\star\gamma}^{m,l} (\bfg\circ\Psi^{-1}) = \{ \bfh\circ\Psi^{-1}\ :\ \bfh\in \A_{\gamma}^{m,l}(\bfg) \}.
	\label{eq:pushfwdAgamma}
    \end{align}
\end{proposition}

\begin{remark}
    In other words, when a tensor $\gamma$ is reconstructible from power densities, then so is any tensor of the form $\Psi_\star \gamma$ with $\Psi$ defined as above. Moreover, if $(g_1,\dots,g_m,h_1,\dots,h_l)$ are boundary conditions on $\partial X$ whose corresponding solutions allow to reconstruct $\gamma$ via the above explicit algorithms, then one may pick precisely $(g_1,\dots,g_m,h_1,\dots,h_l)\circ\Psi^{-1}$ as boundary conditions on $\partial (\Psi(X))$ to reconstruct $\Psi_\star \gamma$. In particular, if $\Psi$ fixes the boundary $\partial X$, then one may pick the same boundary conditions as $\gamma$ to reconstruct $\Psi_\star\gamma$.
\end{remark}

\begin{proof}[Proof of proposition \ref{prop:pushfwd}:] Let $\bfg\in \G_\gamma^m$ for $m\ge n$. The corresponding solutions $\{u_i\}_{i=1}^m$ are being pushforwarded to functions $v_i = u_i\circ\Psi^{-1}$ over $\Psi(X)$ whose power densities are denoted $H'_{ij} = \nabla v_i\cdot[\Psi_\star\gamma]\nabla v_j$. For this proof, primed quantities will always indicate quantities referring to the push-forwarded problem. Using the chain rule and the definition of $\Psi_\star \gamma$, we have the transformation law
    \begin{align}
	H_{ij} (x) = |J_\Psi(x)| H'_{ij} (\Psi(x)), \quad x\in X.
	\label{eq:pushfwdpowdens}
    \end{align}
    Since the functions $\D_\gamma^m[\bfg]$ defined in \eqref{eq:positivitym} are homogeneous polynomials of power densities of degree $n$, we have the following relation
    \begin{align}
	\D_\gamma^m[\bfg](x) = |J_\Psi(x)|^n \D_{\Psi_\star \gamma}^m [\bfg\circ\Psi^{-1}](\Psi(x)), \quad x\in X.
	\label{eq:pushfwdDgamma}
    \end{align}
    By virtue of condition \eqref{eq:psiOK}, the left-hand side of \eqref{eq:pushfwdDgamma} is uniformly bounded away from zero if and only if the right-hand side is as well, which concludes the proof of \eqref{eq:pushfwdGgamma}. 

    On to the proof of \eqref{eq:pushfwdAgamma}, we first look at how things are being push-forwarded locally. As in the preliminaries before definition \ref{def:Agammamk}, an open over $\O = \{\Omega_k\}_{k=1}^K$ of $X$ yields an open cover $\{\Psi(\Omega_k)\}_{k=1}^K$ of $\Psi(X)$ with the same indexing function $I$. This is because of the transformation law
    \begin{align*}
	\det (\nabla_x u_{i_1},\dots,\nabla_x u_{i_n})(x) = J_\Psi(x) \det (\nabla_y v_{i_1}, \dots, \nabla_y v_{i_n}) (\Psi(x)), \quad x\in X,
    \end{align*}
    which ensures that $\{\nabla_x u_i\}_{i\in I(k)}$ is a basis over $\Omega_k$ iff $\{\nabla_y v_i\}_{i\in I(k)}$ is a basis over $\Psi(\Omega_k)$ with $v_i = u_i\circ\Psi^{-1}$. Using \eqref{eq:pushfwdpowdens} and the chain rule, the matrices $Z_{\alpha(k)}$ defined in \eqref{eq:Vomegak} admit the transformation law
    \begin{align}
	Z_{\alpha(k)} = D\Psi^T\ Z'_{\alpha(k)}\circ\Psi \quad (\Omega_k).
	\label{eq:pushfwdZ}
    \end{align}
    In the definition \eqref{eq:Vomegak} of the space $\V_\alpha|_{\Omega_k}$, the scalar function $|J_\Psi|$ appearing from the fact that $H_{(k)}(x) = |J_\Psi(x)| H'_{(k)}(\Psi(x))$ may be absorbed by the space $A_n(\Rm)$, so that we may write
    \begin{align*}
	\V_\alpha|_{\Omega_k} = D\Psi^T\ \V'_{\alpha}|_{\Omega_k}\circ\Psi \quad (\Omega_k). 
    \end{align*}
    Thus the family $\M|_{\Omega_k}$ \eqref{eq:Mp}, from the elements of which one construct cross-products, transforms as
    \begin{align*}
	\M|_{\Omega_k} = D\Psi^T\ \M'|_{\Psi(\Omega_k)}\circ\Psi \quad (\Omega_k).
    \end{align*}
    Using formula \eqref{eq:Ntransf}, we deduce that for $J\in \I(n^2-1,\#\M)$ and throughout $\Omega_k$ 
    \begin{align*}
	\N_{(k)}(J) &= \N(M_{(k)j_1},\dots,M_{(k)j_{n^2-1}}) \\
	&= \N (D\Psi^T\ M'_{(k)j_1}\circ\Psi, \dots, D\Psi^T\ M'_{(k)j_{n^2-1}}\circ\Psi) \\
	&= (J_\Psi)^n D\Psi^{-1} \N (M'_{(k)j_1}, \dots, M'_{(k)j_{n^2-1}}) \circ\Psi.
    \end{align*}
    In particular, the function $\F_\gamma^{m,l}[\bfg,\bfh]|_{\Omega_k}$ defined in \eqref{eq:Fgammaml} transforms according to the rule
    \begin{align}
	\F_\gamma^{m,l}[\bfg,\bfh] |_{\Omega_k} = |J_\Psi|^{2n-1-\frac{2}{n}} \F_{\Psi_\star \gamma}^{m,l} [\bfg\circ\Psi^{-1}, \bfh\circ\Psi^{-1}]\circ\Psi.
	\label{eq:pushfwdFgammaml}
    \end{align}
    Again, by virtue of \eqref{eq:psiOK}, the left-hand side of \eqref{eq:pushfwdFgammaml} is bounded away from zero iff the right-hand side is bounded away from zero. Taking the minimum over $1\le k\le K$ does not change this property, thus \eqref{eq:pushfwdAgamma} is proved. 
\end{proof}

\appendix

\section{Linear algebra}

\subsection{Relations of linear dependence}\label{app:lindep}

\begin{lemma} \label{lem:lindep}  
    Let $(V_1,\ldots, V_{n+1})$ be $n+1$ vectors in $\Rm^n$, and denote $H_{ij} = V_i\cdot V_j$ for $1\le i,j\le n+1$. Then the following linear dependence relation $\sum_{i=1}^{n+1} \mu_i V_i = 0$ holds with coefficients
    \begin{align}
	\begin{split}
	    \mu_i &= -\det (V_1,\dots,V_n)\cdot\det (V_1,\dots, \underbrace{V_{n+1}}_i, \dots, V_n), \\
	    &= (-1)^{i+n+1} \det \{H_{pq}\ |\ 1\le p\le n, 1\le q\le n+1, q\ne i\}, \quad 1\le i\le n, \\
	    \text{and} \quad\mu_{n+1} &= \det (V_1,\dots,V_n)^2 = \det \{H_{ij}\}_{1\le i,j\le n}.	    
	\end{split}	
	\label{eq:lindep}
    \end{align}
\end{lemma}

\begin{proof}
    Define the $\mu_i$'s as in the statement of the function and let us show that $\sum_{i=1}^{n+1}\mu_i V_i = 0$. Consider the vector field defined by the following formal $(n+1) \times (n+1)$ determinant
    \begin{align*}
	V = \det \left(
	\begin{array}{cccc}
	    V_1\cdot V_1 & \cdots & V_1\cdot V_n & V_1\cdot V_{n+1} \\
	    \vdots & \ddots & \vdots & \vdots \\
	    V_n\cdot V_1 & \cdots & V_n\cdot V_n & V_n\cdot V_{n+1} \\
	    V_1 & \cdots & V_n & V_{n+1}
	\end{array}
	\right),
    \end{align*}
    i.e. computed by expanding along the last row. Then we have 
    \begin{align*}
	V &= \sum_{i=1}^{n+1} (-1)^{i+n+1} \det \left( \{H_{pq}\}_{1\le p\le n, 1\le q\le n+1, q\ne i} \right)\ V_i \\
	&= \sum_{i=1}^{n+1} (-1)^{i+n+1} \det (V_1,\dots,V_n)\ \det (V_1,\dots,V_{\hat i}, \dots, V_{n+1})\ V_i \\
	&= - \sum_{i=1}^{n} \det (V_1,\dots,V_n)\cdot \det (V_1,\dots, \underbrace{V_{n+1}}_i, \dots, V_n) V_i + \det(V_1,\dots,V_n)^2 V_n = \sum_{i=1}^{n+1} \mu_i V_i,
    \end{align*}
    where moving $V_{n+1}$ back to the $i$-th position in the $i$-th requires $n-i$ sign flips. We now show that $V=0$. For $1\le i\le n$, the dotproduct $V\cdot V_i$ becomes a determinant of a matrix whose rows of indices $i$ and $n+1$ are equal, therefore $V\cdot S_i=0$. Moreover, $V\cdot S_{n+1}$ is nothing but the determinant of the Gramian matrix of $(V_1,\dots,V_{n+1})$, which is zero since $n+1$ vectors are necessarily linearly dependent. Concluding, we have 
    \begin{align*}
	V\cdot V = \sum_{i=1}^{n+1} \mu_i V\cdot V_i = 0,
    \end{align*}
    thus $V=0$, hence the lemma. 
\end{proof}

\subsection{Generalization of the cross-product} \label{app:crossprod}

Let us consider a $N$-dimensional inner product space $(\V,\dprod{}{})$ with a basis $(\bfe_1,\cdots,\bfe_N)$. Given a linearly independent family of $N-1$ vectors $(V_1,\cdots,V_{N-1})$ in $\V$, a (non-normalized) {\em normal} to the hyperplane spanned by $(V_1,\cdots,V_{N-1})$ is given by computing the formal $\V$-valued determinant
\begin{align}
    \N(V_1,\cdots,V_{N-1}) := \frac{1}{\det (\bfe_1,\cdots,\bfe_N)} \left|
    \begin{array}{ccc}
	\dprod{V_1}{\bfe_1} & \cdots & \dprod{V_1}{\bfe_N} \\
	\vdots & \ddots & \vdots \\
	\dprod{V_{N-1}}{\bfe_1} & \cdots & \dprod{V_{N-1}}{\bfe_N} \\
	\bfe_1 & \cdots & \bfe_N
    \end{array}
    \right|,
    \label{eq:crossprod}
\end{align}
to be expanded along the last row. The function $\N$ can be easily seen to be $N-1$-linear and alternating, and its definition does not depend on the choice of basis $(\bfe_1,\cdots,\bfe_N)$. Moreover, $\N$ satisfies the orthogonality property 
\begin{align*}
    \dprod{\N(V_1,\cdots,V_{N-1})}{V_j} = 0, \quad 1\le j\le N-1,
\end{align*}
as such dotproducts take the form of determinants with identical $j$-th and $N$-th rows. The first important property is that the squared norm of $\N$ represents the hypervolume spanned by $V_1,\dots,V_{N-1}$: 
\begin{align}
    \dprod{\N}{\N} = \det \{\dprod{V_i}{V_j}\}_{1\le i,j\le N-1},
    \label{eq:crossprodnorm}
\end{align}    

We now derive transformation rules when using linear transformations. For $L:\V\to\V$ an automorphism, the following proposition relates $\N (V_1,\cdots,V_{N-1})$ with $\N (LV_1, \cdots, LV_{N-1})$. 

\begin{proposition}\label{prop:crossprodaut}
    For $L\in \text{Aut}(\V)$ and $(V_1, \cdots, V_{N-1})$ a family of linearly independent vectors, the operator $\N$ defined in \eqref{eq:crossprod} satisfies the transformation rule
    \begin{align}
	\N (LV_1, \cdots, LV_{N-1}) = (\det L) L^{-T} \N(V_1,\cdots,V_{N-1}). 
	\label{eq:crossprodaut}
    \end{align}
\end{proposition}

\begin{proof}
    Direct computations yield, picking a basis $(\bfe_1,\cdots,\bfe_N)$ 
    \begin{align*}
	\N(LV_1,\cdots,LV_{N-1}) &= \frac{1}{\det(\bfe_1,\cdots,\bfe_N)} \left|
	\begin{array}{ccc}
	    \dprod{LV_1}{\bfe_1} & \cdots & \dprod{LV_1}{\bfe_N} \\ 
	    \vdots & \vdots & \vdots \\
	    \dprod{LV_{N-1}}{\bfe_1} & \cdots & \dprod{LV_{N-1}}{\bfe_N} \\ 
	    \bfe_1 & \cdots & \bfe_N
	\end{array}
	\right| \\
	&= \frac{\det L}{\det (L^T\bfe_1,\cdots, L^T\bfe_N)} L^{-T} \left|
	\begin{array}{ccc}
	    \dprod{V_1}{L^T\bfe_1} & \cdots & \dprod{V_1}{L^T\bfe_N} \\ 
	    \vdots & \vdots & \vdots \\
	    \dprod{V_{N-1}}{L^T\bfe_1} & \cdots & \dprod{V_{N-1}}{L^T\bfe_N} \\ 
	    L^T\bfe_1 & \cdots & L^T\bfe_N
	\end{array}
	\right|,
    \end{align*}
    where we recognize $\N(V_1,\cdots,V_{N-1})$ expressed in the basis $(L^{T}\bfe_1,\cdots,L^{T}\bfe_N)$, hence the result. 
\end{proof}

We are now interested in the case where $\V = \M_n(\Rm)$ with the inner product $\dprod{M_1}{M_2} = \tr (M_1 M_2^T)$, and where the automorphism $L_A$ denotes left-multiplication by a non-singular matrix $A$. First of all, it is straightforward to see that $L^T_A = L_{A^T}$ and $L_A^{-1} = L_{A^{-1}}$, where $^T$ and $^{-1}$ on the right-hand side denote regular matrix transposition and inversion.

With $(\bfe_1,\cdots,\bfe_n)$ the canonical basis of $\Rm^n$, the family $E_{ij} = \bfe_i\otimes \bfe^j$ for $1\le i,j\le n$ is an orthonormal basis for $M_n(\Rm)$ and we define the orientation on $\M_n(\Rm)$ by
\begin{align*}
    \det (E_{11},\cdots,E_{n1}, \cdots, E_{1n}, \cdots, E_{nn}) = 1.
\end{align*}
Now, if we represent the vectors $AE_{ij}$ in the above oriented basis, we see that 
\begin{align*}
    \det L_A = \det_{\M_n(\Rm)} (AE_{11},\cdots,AE_{n1}, \cdots, AE_{1n}, \cdots, AE_{nn}) = \det \left|
    \begin{array}{ccc}
	A & 0 & 0 \\ 0 & \ddots & 0 \\ 0 & 0 & A
    \end{array}
    \right| = (\det_{\Rm^n} A)^n.
\end{align*}

This brings us to the relation of interest:
\begin{corollary}\label{cor:matrel}
    For $(M_{1}, \dots, M_{n^2-1})\in M_n(\Rm)$, $A\in Gl_n(\Rm)$ and $\N$ defined as in \eqref{eq:crossprod}, we have the following transformation rule:
    \begin{align}
	\N(AM_1,\dots,AM_{n^2-1}) = (\det A)^n A^{-T} \N(M_1,\dots,M_{n^2-1}).
	\label{eq:Ntransf}
    \end{align}
\end{corollary}


\bibliographystyle{siam}
\bibliography{../bibliography/bibliography}

\end{document}